\documentclass[oneside,english]{amsart}
\usepackage[T1]{fontenc}
\usepackage[latin9]{inputenc}
\usepackage{amsthm}
\usepackage{amstext}
\usepackage{amssymb}
\usepackage{wasysym}
\usepackage{esint}

\makeatletter
\numberwithin{equation}{section}
\numberwithin{figure}{section}

\theoremstyle{plain}
  \newtheorem{thm}{\protect\theoremname}[section]
  \newtheorem*{thm*}{\protect\theoremname}
  \newtheorem{lem}[thm]{\protect\lemmaname}
  
\theoremstyle{definition}
  \newtheorem{prop}[thm]{\protect\propname}
\theoremstyle{remark}
  \newtheorem*{rem}{\protect\remarkname}
 \newtheorem*{acknowledgements}{Acknowledgements}
\makeatother
\usepackage{babel}
  \providecommand{\lemmaname}{Lemma}
  \providecommand{\theoremname}{Theorem}
  \providecommand{\propname}{Proposition}
  \providecommand{\remarkname}{Remark}
  \providecommand{\corname}{Corollary}

\begin{document}

\title[Twisted Quadratic Moment]{On the twisted quadratic moment for Dirichlet L-functions}

\author{Seok Hyeong Lee} \address{National Institute for Mathematical Sciences, Daejeon 305-811, South Korea}\email{lshyeong@nims.re.kr}
\author{Seungjai Lee} \address{National Institute for Mathematical Sciences, Daejeon 305-811, South Korea}\email{sjlee@nims.re.kr}

\subjclass[2010]{Primary 11M20}	
\keywords{$L$-function, character, mean values, moments}
\thanks{This work was supported by the National Institute for Mathematical Sciences (NIMS) grant funded by the Korean government (C21602)}
\date{\today}	 
\begin{abstract}
Given $c,$ a positive integer, we give an explicit formula and an
asymptotic formula for 
\[
\sum\chi(c)\left|L(1,\,\chi)\right|^{2},
\]
where $\chi$ is the non-trivial Dirichlet character mod $f$ with $f>c.$
\end{abstract}
\maketitle
\section{Introduction and statement of Results}

Let $c>$1 be a given
positive integer. Let $f>c$ be an integer with gcd$(f,\,c)=1.$  For all non-trivial Dirichlet character $\chi$ mod $f$,  we consider the following sum
\[
M\left(f,\,c\right) :=\sum\chi(c)\left|L(1,\,\chi)\right|^{2}.\]
This can be thought as a twisted version of quadratic moment $\sum\left|L(1,\,\chi)\right|^{2}$, whose asymptotic
\[
\sum\left|L(1,\,\chi)\right|^{2}=\frac{\pi^{2}}{6}\phi(f)\prod_{p\vert f}\left(1-\frac{1}{p^{2}}\right)-\frac{\phi(f)^{2}}{f^{2}}\left(\log f+\sum_{p\vert f}\frac{\log p}{p-1}\right)^{2}+o(\log\log f).
\] was studied by Zhang (\cite{key-9}, \cite{key-10}).

In this paper, we aim to provide the analogous formula for the twisted quadratic moment $M(f,\,c)$.
\begin{thm} We have
	\label{thm:asymp overall}	
	\[
	M(f,c)=\frac{\pi^{2}}{6c}\phi(f)\prod_{p\vert f}\left(1-\frac{1}{p^{2}}\right)-\frac{\phi(f)^{2}}{f^{2}}\left(\log f+\sum_{p\vert f}\frac{\log p}{p-1}\right)^{2}+O(\frac{c^3}{f} + \log f).
	\] 
\end{thm}

Previous work on $M(f,\,c)$ has been done mainly by Louboutin who calculated in \cite{key-5} the explicit formula for the twisted quadratic moment for not all $\chi$ but only the odd ones. To be precise, let $X_{f}^{-}:=\{\chi:\chi(-1)=-1\}$ and $X_{f}^{+}:=\{\chi:\chi(-1)=1\}$ denote
the set of the odd and even Dirichlet characters mod $f$ respectively. Also let $\chi_0$ denote the trivial character mod $f$.
Set 
\begin{alignat*}{1}
M_{-}\left(f,\,c\right) & :=\sum_{\chi\in\chi_{f}^{-}}\chi(c)\left|L(1,\,\chi)\right|^{2},\\
M_{+}\left(f,\,c\right) & :=\sum_{\substack{\chi\in\chi_{f}^{+}\\
		\chi\neq\chi_{0}
	}
}\chi(c)\left|L(1,\,\chi)\right|^{2}.
\end{alignat*}
Louboutin proved the following result {(}{\cite{key-5} Theorem 1}{)}: 
\[
M_{-}\left(f,\,c\right)=\frac{\pi^{2}}{12c}\frac{\phi(f)^{2}}{f}\left(\prod_{p\mid f}\left(1+\frac{1}{p}\right)-\frac{3c}{f}\right)-\frac{\pi^{2}\phi(f)}{4cf^{2}}\sum_{d\mid f}d\mu\left(f/d\right)S\left(c,\,d\right),
\]
where $S(c,\,d)$ is defined by 
\[
S(c,\,d):=\sum_{a=1}^{c-1}\cot\left(\frac{\pi a}{c}\right)\cot\left(\frac{\pi ad}{c}\right)\;\;(\gcd(c,\,d)=1).
\]

Since $M\left(f,\,c\right)= M_{-}\left(f,\,c\right) + M_{+}\left(f,\,c\right),$ it is natural for one to study the even characters, $M_{+}(f,\,c)$ to prove Theorem 1.1. Thus we state and prove the following results:
\begin{thm}
	\label{thm:explicit+}
	We have 
	\[
	\begin{alignedat}{1}M_{+}\left(f,\,c\right) & =\frac{\phi(f)}{f^{2}}\sum_{d\mid f}\mu\left(\frac{f}{d}\right)d\left(\log\left(\frac{f}{d}\right)^{2}+R(c,\,d)\right) -\left(\frac{\phi(f)}{f}\sum_{p\vert f}\frac{\log p}{p-1}\right)^{2}\end{alignedat}
	\]
	where $R(c,\,d)$ is defined by 
	\[
	R(c,\,d):=\sum_{k=1}^{d-1}\log\left(2\sin(\pi k/d)\right)\log\left(2\sin(\pi ck/d)\right).
	\]	
\end{thm}

\begin{thm}
	\label{thm:asymp+}
	\[
	M_{+}(f,c)=\frac{\pi^{2}}{12c}\phi(f)\prod_{p\vert f}\left(1-\frac{1}{p^{2}}\right)-\frac{\phi(f)^{2}}{f^{2}}\left(\log f+\sum_{p\vert f}\frac{\log p}{p-1}\right)^{2}+O(\frac{c^3}{f} +\log f).
	\]
\end{thm}

\section{Preliminary calculations}

In this section, we demonstrate essential preliminary calculations
which will be used repeatedly in the later sections.

\subsection{Formula for the value $L(1,\chi)$}

Let $\chi$ denote the non-trivial Dirichlet character mod $f.$ Our
goal is to show the following. \begin{lem}Let $\zeta_{f}=e^{2\pi i/f}.$
We have 
\[
\begin{aligned}L\left(1,\,\chi\right)=\begin{cases}
-\frac{1}{f}\sum_{a=1}^{f-1}\chi(a)\sum_{j=1}^{f-1}\zeta_{f}^{aj}\log\left(2\sin\frac{\pi j}{f}\right) & \text{if }\chi\in X_{f}^{+},\,\chi\neq\chi_0\\
\frac{\pi}{2f}\sum_{a=1}^{f-1}\chi(a)\cot\frac{\pi a}{f} & \text{if }\chi\in X_{f}^{-}.
\end{cases}\end{aligned}
\]

\end{lem} 
\begin{proof} First, note that 
\[
L(1,\,\chi)=\int_{0}^{1}\frac{\sum_{a=1}^{f}\chi(a)x^{a-1}}{1-x^{f}}\text{d}x.
\]
By using the formula for partial fractions 
\[
\frac{f(X)}{\prod_{i=1}^{n}(X-\alpha_{i})}
=\sum_{i=1}^{n}  \frac{f(\alpha_i)}{ \prod_{j \neq i} (\alpha_j - \alpha_i)} \frac{1}{X-\alpha_{i}}
\]
where $\alpha_{1},\alpha_{2},\cdots,\alpha_{n}$ are all different
and $f$ is polynomial of degree $<n$, we have 
\[
\begin{aligned}\frac{\sum_{a=1}^{f}\chi(a)x^{a-1}}{1-x^{f}} & =\frac{1}{f}\sum_{a=1}^{f-1}\chi(a)\sum_{j=1}^{f-1}\frac{\zeta_{f}^{(a-1)j}}{1-\zeta_{f}^{-j}x}.\end{aligned}
\]
This gives 
\[
L(1,\chi)=\int_{0}^{1}\frac{\sum_{a=1}^{f}\chi(a)x^{a-1}}{1-x^{f}}\text{d}x=\frac{1}{f}\sum_{a=1}^{f-1}\chi(a)\int_{0}^{1}\sum_{j=1}^{f-1}\frac{\zeta_{f}^{(a-1)j}}{1-\zeta_{f}^{-j}x}\text{d}x.
\]
Take the branch cut along the negative real axis. One can easily see
that 
\[
\begin{aligned}\int_{0}^{1}\frac{\text{d}x}{1-\zeta_{f}^{-j}x} & =-\zeta_{f}^{j}\log\left(1-\zeta_{f}^{-j}\right)\end{aligned}
=-\zeta_{f}^{j}\log\left(2\sin\frac{\pi j}{f}\right)+i\zeta_{f}^{j}\left(\frac{\pi j}{f}-\frac{\pi}{2}\right).
\]

Combining all together, we get 
\[
\begin{aligned}L\left(1,\,\chi\right) & =\frac{1}{f}\sum_{a=1}^{f-1}\chi(a)\zeta_{f}^{(a-1)j}\left(\sum_{j=1}^{f-1}-\zeta_{f}^{j}\log\left(2\sin\frac{\pi j}{f}\right)+i\zeta_{f}^{j}\left(\frac{\pi j}{f}-\frac{\pi}{2}\right)\right)\\
 & =\frac{1}{f}\sum_{a=1}^{f-1}\chi(a)\sum_{j=1}^{f-1}\zeta_{f}^{aj}\left(-\log\left(2\sin\frac{\pi j}{f}\right)+i\left(\frac{\pi j}{f}-\frac{\pi}{2}\right)\right).
\end{aligned}
\]
For $\chi$ odd $\left(\chi\left(-1\right)=-1\right),$ we have $\chi\left(f-a\right)=-\chi(a).$
Substituting $a'=f-a,$ $j'=f-j$ changes sign of $\chi(a)$ but leaves
$\zeta_{f}^{aj}$ and $\log\left(2\sin\frac{\pi j}{f}\right)$ the
same. Hence 
\[
L\left(1,\,\chi\right)=-\frac{1}{f}\sum_{a=1}^{f-1}\chi(a)\sum_{j=1}^{f-1}\zeta_{f}^{aj}\cdot i\left(\frac{\pi}{2}-\frac{\pi j}{f}\right).
\]
By expanding this out, 
\[
\begin{aligned}L(1,\,\chi) & =-\frac{1}{f}\sum_{a=1}^{f-1}\chi(a)\left(\frac{i\pi}{2}+\frac{\pi}{1-\zeta_{f}^{a}}\right)=-\frac{i\pi}{f}\sum_{a=1}^{f-1}\chi(a)\left(1-\zeta_{f}^{a}\right)^{-1}\\
 & =-\frac{i\pi}{f}\sum_{a=1}^{f-1}\frac{\chi(a)\left(1-\zeta_{f}^{-a}\right)}{\left|1-\zeta_{f}^{a}\right|^{2}}=-\frac{i\pi}{f}\sum_{a=1}^{f-1}\chi(a)\frac{\left(1-\cos\frac{2\pi a}{f}+i\sin\frac{2\pi a}{f}\right)}{4\sin^{2}\frac{\pi a}{f}}.
\end{aligned}
\]
Again, changing $a'=f-a$ cancels out the part $(1-\cos(2\pi a/f))/(4\sin^{2}\pi a/f)$.
So the remaining part gives 
\[
L(1,\,\chi)=-\frac{i\pi}{f}\sum_{a=1}^{f-1}\chi(a)\frac{\left(i\sin\frac{2\pi a}{f}\right)}{4\sin^{2}\frac{\pi a}{f}}=\frac{\pi}{2f}\sum_{a=1}^{f-1}\chi(a)\cot\left(\frac{\pi a}{f}\right).
\]

For $\chi$ even, note that this time only the sign of $\frac{\pi}{2}-\frac{\pi j}{f}$
changes. By the similar arguments, the parts involving $i\left(\frac{\pi}{2}-\frac{\pi j}{f}\right)$
vanish. Therefore we have 
\[
L\left(1,\,\chi\right)=-\frac{1}{f}\sum_{a=1}^{f-1}\chi(a)\sum_{j=1}^{f-1}\zeta_{f}^{aj}\log\left(2\sin\frac{\pi j}{f}\right).
\] \end{proof}
\begin{rem} One can check that our result for odd $\chi$ is identical with Louboutin's result in \cite{key-3}, Proposition 1. 
\end{rem}

\subsection{Asymptotic for harmonic series}

We first note
\[
\sum_{|n|<X,n\neq0}\frac{1}{|n|}=2\log X+2\gamma+O(\frac{1}{X})
\]
where $\gamma$ is the Euler-Mascheroni constant.

\begin{prop}  Define the function $F:\mathbb{R}-\mathbb{Z}\rightarrow\mathbb{R}$ as 
	\[
	F(x):=\sum_{-\infty<n<\infty}(\frac{1}{|n-x|}-\frac{1}{|n|})+\frac{1}{|x|}.
	\]  
	For $a>0$ and $b$ not divisible by $a$, we have 
	\[
	\sum_{|n|<X,n\equiv b(\mathrm{mod}\,a)}\frac{1}{|n|}=\frac{2\log X-2\log a+2\gamma}{a}+\frac{1}{a}F(\frac{b}{a})+O(\frac{1}{X}),
	\]	
	where the implied constant (for $O$) does not depend on $a,b,X$. \end{prop}
\begin{proof}
First, we discuss the well-definedness of $F(x)$.  The summand has size
\[
\frac{1}{|n-x|}-\frac{1}{|n|} = \frac{|n|- |n-x|}{|n||n-x|} = O(\frac{|x|}{n^2})
\]
for $n>2|x|$ so the series converges, and we can also see
\[
F(x) = \sum_{|n|<X,n \neq 0} (\frac{1}{|n-x|}-\frac{1}{|n|}) + \frac{1}{|x|} + O(\frac{|x|}{X}).
\]
Now, we observe
\[
\sum_{|n|<X, n \equiv 0 (\mathrm{mod}\,a), n \neq 0} \frac{1}{|n|} = \sum_{|m|<X/a, m \neq 0} \frac{1}{|am|} = \frac{2\log X - 2 \log a + 2 \gamma}{a} + O (\frac{1}{X}).
\]
So it suffices to show
\[
\sum_{|n|<X, n \equiv b (\mathrm{mod}\,a)} \frac{1}{|n|}- \sum_{|n|<X, n \equiv 0 (\mathrm{mod}\,a), n \neq 0} \frac{1}{|n|} = \frac{1}{a} F(\frac{b}{a}) + O(\frac{1}{X}).
\]
Let $n=ma+b$ and $n'=ma$. One can represent their differences as
\begin{align*}
& \sum_{|n|<X, n \equiv b (\mathrm{mod}\,a)} \frac{1}{|n|} - \sum_{|n'|<X, n' \equiv 0 (\mathrm{mod}\,a), n' \neq 0} \frac{1}{|n'|} \\
&= \sum_{|m|<X/a, m \neq 0} \left( \frac{1}{|ma+b|} - \frac{1}{|ma|} \right) + \frac{1}{b} + O (\frac{1}{X}) \\
&= \frac{1}{a}\left(  \sum_{|m|<X/a, m \neq 0} \left(\frac{1}{|m+\frac{b}{a}|} - \frac{1}{|m|} \right) + \frac{1}{b/a} \right) + O(\frac{1}{X}) \\
&= \frac{1}{a} F(\frac{b}{a}) + O(\frac{1}{X}).
\end{align*}Hence our claim holds.
\end{proof}
\begin{rem}  One can see that if we
	further define $F(n)=0$ for $n\in\mathbb{Z}$, then the formula above also holds
	for $a\vert b$.
\end{rem}
Throughout the rest of this paper, let $F(x)$ be the function as defined in Proposition 2.2.

\begin{prop} The function $F$ satisfies the following property:
\\
 (a) $F(x+1)=F(x)=F(1-x)$. \\
 (b) For $0<x<1$, if we define $G(x)$ to be 
\[
F(x)=\frac{1}{x}+\frac{1}{1-x}-1+G(x)
\]
then $|G(x)|=O(x(1-x))$. \end{prop}

\begin{proof} (a) It can be checked easily by arranging the terms
appropriately. \\
 (b) For $0<x<1$ we have 
\begin{align*}
F(x) & =\frac{1}{x}+\frac{1}{1-x}+\frac{1}{1+x}+\sum_{n\ge2}(\frac{1}{n-x}+\frac{1}{n+x}-\frac{2}{n})\\
 & =\frac{1}{x}+\frac{1}{1-x}-1+\sum_{n\ge2}\frac{2x^{2}}{n(n^{2}-x^{2})}-\frac{x}{1+x}.
\end{align*}
So it is immediate that 
\[
G(x)=\sum_{n\ge2}\frac{2x^{2}}{n(n^{2}-x^{2})}-\frac{x}{1+x}
\]
can be extended to well-defined function on $(-1/2,3/2)$ and $G(0)=0$.
Also, as we subtracted symmetric terms, $G(x)$ too satisfies $G(x)=G(1-x)$. Hence $G(1)=0$, and $|G(x)|=O(x(1-x))$ follows. \end{proof}

\begin{rem} On $0<x<1$, The function $F$ can be represented using
the digamma function $\psi(z)=\Gamma'(z)/\Gamma(z)$ as 
\[
F(x)=-2\gamma+\psi(1+x)+\psi(1-x)+\frac{1}{x}.
\]
\end{rem}

\begin{prop} We have 
\[
\sum_{0<a<d}F(\frac{a}{d})=2d\log d.
\]
\end{prop}

\begin{proof} We consider 
\[
\sum_{|n|<X,n\neq0}\frac{1}{|n|}=\sum_{0\le a<d}\sum_{|n|<X,n\equiv a(d)}\frac{1}{|n|}.
\]
By Proposition 2.2, the left hand side becomes $2\log X+2\gamma+O(1/X)$ and the right hand side becomes 
\begin{align*}
 & \sum_{0\le a<d}\left(\frac{2\log X-2\log d+2\gamma}{d}+\frac{1}{d}F(\frac{a}{d})+O(\frac{1}{X})\right)\\
 & =2\log X+2\gamma+\frac{1}{d}\sum_{0<a<d}F(\frac{a}{d})-2\log d+O(\frac{d}{X}).
\end{align*}
This proves our claim. \end{proof}

\subsection{Other useful lemmas}

\begin{lem} (\cite{key-8}) For $\zeta_{f}=e^{2\pi i/f}$, 
\[
\sum_{1\le a\le f}\zeta_{f}^{at}=\begin{cases}
f-1 & \text{if }f\vert t,\\
-1 & \text{otherwise},
\end{cases}
\]and
\[
\sum_{(a,f)=1}\zeta_{f}^{at}=\sum_{d\mid\gcd(t,\,f)}\mu\left(\frac{f}{d}\right)d=\frac{\phi(f)\mu(f/(f,t))}{\phi(f/(f,t))}.
\]
\end{lem}

\begin{lem}(\cite{key-8}) 
\[
\sum_{d\vert f}\frac{\mu(d)}{d}=\frac{\phi(f)}{f},
\]and
\[
\sum_{d\vert f}\frac{\mu(d)\log d}{d}=-\frac{\phi(f)}{f}\sum_{p\vert f}\frac{\log p}{p-1}.
\]
\end{lem}

\begin{lem} \label{logexp} For $|z|=1$ 
\[
\log|1-z|=\sum_{n\neq0}\frac{z^{n}}{2|n|}=\sum_{|n|<X,n\neq0}\frac{z^{n}}{2|n|}+O(\frac{1}{|1-z|X}).
\]
\end{lem}

\section{An exact formula for $M_{+}(f,\,c):$ Proof of Theorem \ref{thm:explicit+}}

In this section, we want to calculate the explicit formula for $M_{+}\left(f,c\right)$. We start from the following proposition.
\begin{prop}
	\label{prop:3.1}
	\[
	M_{+}\left(f,c\right)=\frac{\phi(f)}{f^{2}}\sum_{d\mid f}\mu\left(\frac{f}{d}\right)d\sum_{\substack{\substack{j,k=1\\
				d\mid j-ck
			}
		}
	}^{f-1}\log\left(2\sin\frac{\pi j}{f}\right)\log\left(2\sin\frac{\pi k}{f}\right)-\left(\frac{\phi(f)}{f}\sum_{p\vert f}\frac{\log p}{p-1}\right)^{2}.
	\]
\end{prop}
\begin{proof}
First, recall Lemma 2.1, 
\[
L\left(1,\,\chi\right)=-\frac{1}{f}\sum_{a=1}^{f-1}\chi(a)\sum_{j=1}^{f-1}\zeta_{f}^{aj}\log\left(2\sin\frac{\pi j}{f}\right)
\]
for $\chi$ even. This gives \[
\begin{aligned}
 \sum_{\substack{\chi\in X_{f}^{+}\\
\chi\neq1
}
}\chi\left(c\right)\left|L\left(1,\,\chi\right)\right|^{2}= &\sum_{\chi\in X_{f}^{+}}\chi(c)\left|\frac{1}{f}\sum_{a=1}^{f-1}\chi(a)\sum_{j=1}^{f-1}\zeta_{f}^{aj}\log\left(2\sin\frac{\pi j}{f}\right)\right|^{2} \\
&-\chi_{0}\left(c\right)\left|\frac{1}{f}\sum_{a=1}^{f-1}\chi_{0}(a)\sum_{j=1}^{f-1}\zeta_{f}^{aj}\log\left(2\sin\frac{\pi j}{f}\right)\right|^{2}.
\end{aligned}\]
To calculate the value for $\chi_{0}$, note that 
\begin{align*}
\sum_{j=1}^{m-1}\log\left(2\sin\frac{\pi j}{m}\right) & =\sum_{j=1}^{m-1}\log\left|1-\zeta_{m}^{j}\right|\\
 & =\log m.
\end{align*}
This with Lemma 2.5 implies 
\[
\begin{aligned}\frac{1}{f}\sum_{a=1}^{f-1}\chi_{0}(a)\sum_{j=1}^{f-1}\zeta_{f}^{aj}\log\left(2\sin\frac{\pi j}{f}\right) & =\frac{1}{f}\sum_{j=1}^{f-1}\log(2\sin\frac{\pi j}{f})\sum_{(a,f)=1}\zeta_{f}^{aj}\\
 & =\frac{1}{f}\sum_{j=1}^{f-1}\log(2\sin\frac{\pi j}{f})\sum_{d\mid f,\,j}\mu\left(\frac{f}{d}\right)d\\
 & =\frac{1}{f}\sum_{d\mid f}\mu\left(\frac{f}{d}\right)d\sum_{\substack{0<j<f\\
d\mid j
}
}\log\left(2\sin\frac{\pi j}{f}\right)\\
 & =\frac{1}{f}\sum_{d\mid f}\mu\left(\frac{f}{d}\right)d\log\left(\frac{f}{d}\right)\\
 & =-\frac{\phi(f)}{f}\sum_{p\vert f}\frac{\log p}{p-1}.
\end{aligned}
\]
So we have 
\begin{align*}M_{+}\left(f,c\right) & =\sum_{\chi\in X_{f}^{+}}\chi(c)\left|\frac{1}{f}\sum_{a=1}^{f-1}\chi(a)\sum_{j=1}^{f-1}\zeta_{f}^{aj}\log\left(2\sin\frac{\pi j}{f}\right)\right|^{2}-\left(\frac{\phi(f)}{f}\sum_{p\vert f}\frac{\log p}{p-1}\right)^{2}.\end{align*}

Now, we can write
\begin{align*} & \sum_{\chi\in X_{f}^{+}}\chi(c)\left|\frac{1}{f}\sum_{a=1}^{f-1}\chi(a)\sum_{j=1}^{f-1}\zeta_{f}^{aj}\log\left(2\sin\frac{\pi j}{f}\right)\right|^{2}\\
 & =\sum_{\chi\in X_{f}^{+}}\chi(c)\left(\frac{1}{f}\sum_{a=1}^{f-1}\chi(a)\sum_{j=1}^{f-1}\zeta_{f}^{aj}\log\left(2\sin\frac{\pi j}{f}\right)\right)\cdot\left(\frac{1}{f}\sum_{b=1}^{f-1}\overline{\chi(b)}\sum_{k=1}^{f-1}\zeta_{f}^{-bk}\log\left(2\sin\frac{\pi k}{f}\right)\right)\\
 & =\frac{1}{f^{2}}\sum_{\substack{a,\,b,\,j,\,k=1\\
(a,f)=(b,f)=1
}
}^{f-1}\left(\sum_{\chi\in X_{f}^{+}}\chi(a)\overline{\chi(b)}\chi(c)\right)\zeta_{f}^{aj-bk}\log\left(2\sin\frac{\pi j}{f}\right)\log\left(2\sin\frac{\pi k}{f}\right)\\
&=\frac{\phi(f)}{2f^{2}}\sum_{\substack{a,\,j,\,k=1\\
		(a,f)=1
	}
}^{f-1}\left(\zeta_{f}^{aj-ack}+\zeta_{f}^{aj+ack}\right)\log\left(2\sin\frac{\pi j}{f}\right)\log\left(2\sin\frac{\pi k}{f}\right)
\end{align*}

as 
\[\sum_{\chi\in X_{f}^{+}}\chi(a)\overline{\chi(b)}\chi(c)=\begin{cases}
\frac{\phi(f)}{2} & \text{if}\;ac\equiv\pm b\;\left(\text{mod}\;f\right),\\
0 & \text{otherwise.}
\end{cases}\]

One can note that if we replace $k$ by $f-k$, we observe
\[\begin{aligned} & 
\sum_{\substack{a,\,j,\,k=1\\
(a,f)=1
}
}^{f-1}\left(\zeta_{f}^{aj+ack}\right)\log\left(2\sin\frac{\pi j}{f}\right)\log\left(2\sin\frac{\pi k}{f}\right)\\
&
=
\sum_{\substack{a,\,j,\,k=1\\
(a,f)=1
}
}^{f-1}\left(\zeta_{f}^{aj-ack}\right)\log\left(2\sin\frac{\pi j}{f}\right)\log\left(2\sin\frac{\pi k}{f}\right).
\end{aligned}
\]
Hence we only need to consider the parts involving $\zeta_{f}^{aj-ack}$. This allows us to write 
\[
\begin{aligned}
 &\frac{\phi(f)}{2f^{2}}\sum_{\substack{a,\,j,\,k=1\\
(a,f)=1
}
}^{f-1}\left(\zeta_{f}^{aj-ack}+\zeta_{f}^{aj+ack}\right)\log\left(2\sin\frac{\pi j}{f}\right)\log\left(2\sin\frac{\pi k}{f}\right)\\
 & =\frac{\phi(f)}{f^{2}}\sum_{\substack{j,\,k=1}
}^{f-1}\log\left(2\sin\frac{\pi j}{f}\right)\log\left(2\sin\frac{\pi k}{f}\right)\sum_{\substack{(a,f)=1}
}^{f-1}\left(\zeta_{f}^{a(j-ck)}\right)\\
 & =\frac{\phi(f)}{f^{2}}\sum_{\substack{j,\,k=1}
}^{f-1}\log\left(2\sin\frac{\pi j}{f}\right)\log\left(2\sin\frac{\pi k}{f}\right)\left(\sum_{d\mid f,\,j-ck}\mu\left(\frac{f}{d}\right)d\right)\\
 & =\frac{\phi(f)}{f^{2}}\sum_{d\mid f}\mu\left(\frac{f}{d}\right)d\sum_{\substack{\substack{j,k=1\\
d\mid j-ck
}
}
}^{f-1}\log\left(2\sin\frac{\pi j}{f}\right)\log\left(2\sin\frac{\pi k}{f}\right).
\end{aligned}
\]
Thus we get 
\[
M_{+}\left(f,c\right)=\frac{\phi(f)}{f^{2}}\sum_{d\mid f}\mu\left(\frac{f}{d}\right)d\sum_{\substack{\substack{j,k=1\\
d\mid j-ck
}
}
}^{f-1}\log\left(2\sin\frac{\pi j}{f}\right)\log\left(2\sin\frac{\pi k}{f}\right)-\left(\frac{\phi(f)}{f}\sum_{p\vert f}\frac{\log p}{p-1}\right)^{2}.
\]
as required. 	
\end{proof}
Next, we calculate 
\[
\sum_{\substack{j,k=1\\
d\mid j-ck
}
}^{f-1}\log\left(2\sin\frac{\pi j}{f}\right)\log\left(2\sin\frac{\pi k}{f}\right).
\]
\begin{prop}\label{prop:3.2} For $d$ and $f$ such that $d\mid f$,
\[
	\begin{alignedat}{1} & \sum_{\substack{j,k=1\\
			d\mid j-ck
		}
	}^{f-1}\log\left(2\sin\frac{\pi j}{f}\right)\log\left(2\sin\frac{\pi k}{f}\right)\\
	& =\log\left(\frac{f}{d}\right)^{2}+\sum_{1\leq u\leq d-1}\log\left(2\sin\frac{\pi cu}{d}\right)\log\left(2\sin\frac{\pi u}{d}\right).
	\end{alignedat}
	\]
\end{prop}
\begin{proof}
For $k$ fixed and $d\mid k,$ the condition $d\mid j-ck$ is satisfied
for $j=d,\,2d,\,\ldots,\,\left(\frac{f}{d}-1\right)d$. So 
\begin{equation}
\begin{alignedat}{1}\sum_{\substack{1\leq j\leq f-1\\
d\mid j-ck
}
}\log\left(2\sin\frac{\pi j}{f}\right) & =\sum_{\substack{1\leq j\leq f-1\\
d\mid j
}
}\log\left|1-\zeta_{f}^{j}\right|\\
 & =\log\left|\left(1-\zeta_{f}^{d}\right)\left(1-\zeta_{f}^{2d}\right)\cdots\left(1-\zeta_{f}^{f-d}\right)\right|\\
 & =\log\left(\frac{f}{d}\right).
\end{alignedat}
\end{equation}
For $d\nmid k,$ the condition $d\mid j-ck$ is satisfied by $\frac{f}{d}$
values of $j,$ and they can be chosen to be $ck,$ $ck+d,\,\ldots\,,\,ck+\left(\frac{f}{d}-1\right)d$
instead of their reduction by mod $f.$ Thus 
\[
\begin{alignedat}{1}\sum_{\substack{1\leq j\leq f-1\\
d\mid j-ck
}
}\log\left(2\sin\frac{\pi j}{f}\right) & =\sum_{\substack{1\leq j\leq f-1\\
d\mid j-ck
}
}\log\left|1-\zeta_{f}^{j}\right|\\
 & =\log\left|\left(1-\zeta_{f}^{ck}\right)\left(1-\zeta_{f}^{ck}\zeta_{f}^{d}\right)\left(1-\zeta_{f}^{ck}\zeta_{f}^{2d}\right)\cdots\left(1-\zeta_{f}^{ck}\zeta_{f}^{f-d}\right)\right|\\
 & =\log\left|1-\zeta_{f}^{ckf/d}\right|\\
 & =\log\left|1-\zeta_{d}^{ck}\right|.\\
 & =\log\left(2\sin\frac{\pi ck}{d}\right).
\end{alignedat}
\]
Combining this with $\left(3.1\right),$ we get 
\[
\begin{alignedat}{1}\sum_{\substack{j,k=1\\
d\mid j-ck
}
}^{f-1}\log\left(2\sin\frac{\pi j}{f}\right)\log\left(2\sin\frac{\pi k}{f}\right)= & \sum_{\substack{k=1\\
d\mid k
}
}^{f-1}\log\left(2\sin\frac{\pi k}{f}\right)\sum_{\substack{j=1\\
d\mid j-ck
}
}^{f-1}\log\left(2\sin\frac{\pi j}{f}\right)\\
 & +\sum_{\substack{k=1\\
d\nmid k
}
}^{f-1}\log\left(2\sin\frac{\pi k}{f}\right)\sum_{\substack{j=1\\
d\mid j-ck
}
}^{f-1}\log\left(2\sin\frac{\pi j}{f}\right)\\
= & \log\left(\frac{f}{d}\right)^{2}+\sum_{\substack{k=1\\
d\nmid k
}
}^{f-1}\log\left(2\sin\frac{\pi k}{f}\right)\log\left(2\sin\frac{\pi ck}{d}\right).
\end{alignedat}
\]
Now, for 
\[
\sum_{\substack{k=1\\
d\nmid k
}
}^{f-1}\log\left(2\sin\frac{\pi k}{f}\right)\log\left(2\sin\frac{\pi ck}{d}\right),
\]
$k$ again can be grouped according to the remainder mod $d.$ Let $k=u+dv,\;1\leq u\leq d-1,\;0\leq v\leq\frac{f}{d}-1.$
Then 
\[
\begin{alignedat}{1} & \sum_{\substack{k=1\\
d\nmid k
}
}^{f-1}\log\left(2\sin\frac{\pi k}{f}\right)\log\left(2\sin\frac{\pi ck}{d}\right)\\
 & =\sum_{1\leq u\leq d-1}\sum_{0\leq v\leq(f/d)-1}\log\left(2\sin\frac{\pi c\left(u+dv\right)}{d}\right)\log\left(2\sin\left(\frac{\pi u}{f}+\frac{\pi v}{(f/d)}\right)\right)\\
 & =\sum_{1\leq u\leq d-1}\log\left(2\sin\frac{\pi cu}{d}\right)\sum_{0\leq v\leq(f/d)-1}\log\left|1-\zeta_{f}^{u}\zeta_{f/d}^{v}\right|\\
 & =\sum_{1\leq u\leq d-1}\log\left(2\sin\frac{\pi cu}{d}\right)\log\left(2\sin\frac{\pi u}{d}\right).
\end{alignedat}
\]
Hence we get 
\[
\begin{alignedat}{1} & \sum_{\substack{j,k=1\\
		d\mid j-ck
	}
}^{f-1}\log\left(2\sin\frac{\pi j}{f}\right)\log\left(2\sin\frac{\pi k}{f}\right)\\
& =\log\left(\frac{f}{d}\right)^{2}+\sum_{1\leq u\leq d-1}\log\left(2\sin\frac{\pi cu}{d}\right)\log\left(2\sin\frac{\pi u}{d}\right).
\end{alignedat}
\]
as required.
\end{proof}
Finally, by combining Proposition \ref{prop:3.1} and \ref{prop:3.2} we get 
\[
\begin{alignedat}{1}M_{+}\left(f,c\right) & =\frac{\phi(f)}{f^{2}}\sum_{d\mid f}\mu\left(\frac{f}{d}\right)d\sum_{\substack{\substack{j,k=1\\
d\mid j-ck
}
}
}^{f-1}\log\left(2\sin\frac{\pi j}{f}\right)\log\left(2\sin\frac{\pi k}{f}\right)-\left(\frac{\phi(f)}{f}\sum_{p\vert f}\frac{\log p}{p-1}\right)^{2}\\
 & =\frac{\phi(f)}{f^{2}}\sum_{d\mid f}\mu\left(\frac{f}{d}\right)d\left(\log\left(\frac{f}{d}\right)^{2}+\sum_{1\leq u\leq d-1}\log\left(2\sin\frac{\pi cu}{d}\right)\log\left(2\sin\frac{\pi u}{d}\right)\right)\\
 & -\left(\frac{\phi(f)}{f}\sum_{p\vert f}\frac{\log p}{p-1}\right)^{2}.
\end{alignedat}
\]
This proves Theorem \ref{thm:explicit+}.

\section{An asymptotic formula for $M_{+}(f,\,c):$ Proof of Theorem \ref{thm:asymp+}}

First, we have to evaluate the kernel 
\begin{align*}
R(c,d) & =\sum_{k=1}^{d-1}\log\left(2\sin(\pi k/d))\log(2\sin(\pi ck/d)\right)\\
 & =\sum_{k=1}^{d-1}\log|1-\zeta_{d}^{k}|\log|1-\zeta_{d}^{ck}|.
\end{align*}
Using the identity of Lemma \ref{logexp}, we expand out the $\log$
inside 
\begin{align*}
R(c,d) & =\sum_{k=1}^{d-1}\left(\sum_{|m|<X_{1},m\neq0}\frac{\zeta_{d}^{mk}}{2|m|}+O(\frac{1}{X_{1}})\right)\left(\sum_{|n|<X_{2},n\neq0}\frac{\zeta_{d}^{cnk}}{2|n|}+O(\frac{1}{X_{2}})\right)\\
 & =\sum_{\substack{|m|<X_{1},|n|<X_{2}\\mn\neq0}}\;\sum_{k=1}^{d-1}\frac{\zeta_{d}^{(m+cn)k}}{4|mn|}+O\left(\frac{\log X_{1}}{X_{2}}+\frac{\log X_{2}}{X_{1}}\right).
\end{align*}
This implies 
\[
R(c,d)=\sum_{\substack{|m|<X_{1},|n|<X_{2}\\mn\neq0,\,d\vert m+cn}}\frac{d}{4|mn|}-\sum_{\substack{|m|<X_{1},|n|<X_{2}\\mn\neq0}}\frac{1}{4|mn|}+O\left(\frac{\log X_{1}}{X_{2}}+\frac{\log X_{2}}{X_{1}}\right).
\]
Now, we have
\[
\sum_{\substack{|m|<X_{1},|n|<X_{2}\\mn\neq0}}\frac{1}{4|mn|}=(\log X_{1}+\gamma)(\log X_{2}+\gamma)+O\left(\frac{\log X_{1}}{X_{2}}+\frac{\log X_{2}}{X_{1}}\right).
\]
and 
\begin{align*}
 & \sum_{\substack{|m|<X_{1},|n|<X_{2}\\mn\neq0,\,d\vert m+cn}}\frac{d}{4|mn|}\\
 & =\sum_{|n|<X_{2},n\neq0}\sum_{|m|<X_{1},m\neq0,m\equiv-cn(d)}\frac{d}{4|mn|}\\
 & =\sum_{|n|<X_{2},n\neq0}\frac{d}{4|n|}\left(\frac{2\log X_{1}-2\log d+2\gamma}{d}+\frac{1}{d}F(\frac{-cn}{d})+O(\frac{1}{X_{1}})\right)\\
 & =(\log X_{2}+\gamma)(\log X_{1}-\log d+\gamma)+\sum_{|n|<X_{2},d\nmid n}\frac{F(-cn/d)}{4|n|}+O(\frac{\log X_{2}}{X_{1}})\\
 & =(\log X_{2}+\gamma)(\log X_{1}-\log d+\gamma)+\sum_{0<e<d}F(\frac{-ce}{d})\sum_{|n|<X_{2},n\equiv e}\frac{1}{4|n|}+O(\frac{\log X_{2}}{X_{1}})\\
 & =(\log X_{2}+\gamma)(\log X_{1}-\log d+\gamma)+\sum_{0<e<d}F(\frac{-ce}{d})\left(\frac{\log X_{2}-\log d+\gamma}{2d}+\frac{F(e/d)}{4d}\right)\\
 &\quad +O(\frac{\log X_{2}}{X_{1}}+\frac{1}{X_2})\\
 & =(\log X_{1}+\gamma)(\log X_{2}+\gamma)+\frac{1}{4d}\sum_{0<e<d}F(\frac{-ce}{d})F(\frac{e}{d})-(\log d)^{2}+O(\frac{\log X_{2}}{X_{1}}+\frac{1}{X_2})
\end{align*}
as 
\[
\sum_{0<e<d}F(\frac{e}{d})=2d\log d.
\]
Combining these calculations, we get  
\[
R(c,d)=\frac{1}{4d}\sum_{0<e<d}F(\frac{-ce}{d})F(\frac{e}{d})-(\log d)^{2}+O\left(\frac{\log X_{1}+\log X_{2}}{\min(X_{1},X_{2})}\right).
\]
As $F$ is even, by sending $X_{1}$ and $X_{2}$ to infinity we
can write
\[
R(c,d)=\frac{1}{4d}\sum_{0<e<d}F(\frac{ce}{d})F(\frac{e}{d})-(\log d)^{2}.
\]
Now, we will prove the following result:

\begin{lem} Let $d>c$ coprime. 
We have 
\[
R(c,d)=\frac{\pi^{2}}{12}\frac{d}{c}-(\log d)^{2}+O(\log d)
\] with the implied constant not depending on $c$.
\end{lem}

\begin{proof} We recall that for $0<x<1$ 
\[
F(x)=\frac{1}{x}+\frac{1}{1-x}-1+G(x)
\]
then $G(x)=O(x(1-x))$. We have 
\begin{align*}
R(c,d) & =\frac{1}{4d}\sum_{e=1}^{d-1}F(\frac{ce}{d})F(\frac{e}{d})-(\log d)^{2}\\
 & =\frac{1}{4d}\sum_{e=1}^{d-1}(\frac{1}{\{ce/d\}}+\frac{1}{1-\{ce/d\}}-1+G(\frac{ce}{d}))(\frac{1}{e/d}+\frac{1}{1-e/d}-1+G(\frac{e}{d}))-(\log d)^{2}
\end{align*}
where $\{\alpha\}$ is the fractional part of $\alpha$. Note that
both in $\{ce/d\}$ and $1-\{ce/d\}$, elements of $\{1/d,2/d,\cdots,(d-1)/d\}$
appears once. So 
\[
|\frac{1}{4d}\sum_{e=1}^{d-1}(\frac{1}{\{ce/d\}}+\frac{1}{1-\{ce/d\}})(-1+G(\frac{e}{d}))|=O(\frac{1}{4d}\cdot2\sum_{e=1}^{d-1}\frac{d}{e})=O(\log d).
\]
Same holds for the sum 
\[
|\frac{1}{4d}\sum_{e=1}^{d-1}(-1+G(\frac{ce}{d}))(\frac{1}{e/d}+\frac{1}{1-e/d})|=O(\frac{1}{4d}2\sum_{e=1}^{d-1}\frac{d}{e})=O(\log d)
\]
and sum of $G(ce/d)G(d)$ is obviously bounded too. So it remains
to calculate 
\begin{align*}
R(c,d) & =\frac{1}{4d}\sum_{e=1}^{d-1}\frac{1}{e/d}\frac{1}{\{ce/d\}}+\frac{1}{4d}\sum_{e=1}^{d-1}\frac{1}{e/d}\frac{1}{(1-\{ce/d\})}\\
 & +\frac{1}{4d}\sum_{e=1}^{d-1}\frac{1}{(1-e/d)}\frac{1}{\{ce/d\}}+\frac{1}{4d}\sum_{e=1}^{d-1}\frac{1}{(1-e/d)}\frac{1}{(1-\{ce/d\})}-(\log d)^{2}+O(\log d).
\end{align*}
Observe that if we replace $e$ to $d-e$, we have $(e/d)\leftrightarrow(1-e/d)$
and $\{ce/d\}\leftrightarrow\{1-ce/d\}$. Hence we have 
\[
R(c,d)=\frac{1}{2d}\sum_{e=1}^{d-1}\frac{1}{(1-e/d)}\frac{1}{\{ce/d\}}+\frac{1}{2d}\sum_{e=1}^{d-1}\frac{1}{e/d}\frac{1}{\{ce/d\}}-(\log d)^{2}+O(\log d).
\]
Let $p=\lfloor ce/d\rfloor$. Then as $p$ ranges from $0$ to $c-1$, we have 
\[
\frac{1}{2d}\sum_{e=1}^{d-1}\frac{1}{(1-e/d)}\frac{1}{\{ce/d\}}=\frac{1}{2d}\sum_{p=0}^{c-1}\sum_{p<ce/d<p+1}\frac{1}{(1-\frac{e}{d})}\frac{1}{(\frac{ce}{d}-p)}.
\]
For $p<c-1$,
\[
1-\frac{e}{d}>1-\frac{p+1}{c}
\]
and 
\begin{align*}
\sum_{p<ce/d<p+1}\frac{1}{\frac{ce}{d}-p} & <\frac{1}{\frac{1}{d}}+\frac{1}{\frac{c+1}{d}}+\frac{1}{\frac{2c+1}{d}}+\cdots+\frac{1}{\frac{(d/c)c+1}{d}}\\
 & =d(1+\frac{1}{c}\log\frac{d}{c}+\gamma+O(\frac{c}{d}))=O(\frac{d}{c}\log d).
\end{align*}
Thus 
\[
\frac{1}{2d}\sum_{p=0}^{c-2}\sum_{p<\frac{ce}{d}<p+1}\frac{1}{(1-\frac{e}{d})}\frac{1}{(\frac{ce}{d}-p)}=O(\frac{1}{2d}\sum_{p=0}^{c-2}\frac{1}{1-\frac{p+1}{c}}\cdot\frac{d}{c}\log d)=O(\frac{1}{2d}\frac{c}{\log c}\frac{d}{c}\log d)=O(\log d).
\]
For $p=c-1$, let $t=d-e$. Then $0<t<d/c$, and 
\[
\frac{1}{2d}\sum_{c-1<\frac{ce}{d}<c}\frac{1}{(1-\frac{e}{d})}\frac{1}{(\frac{ce}{d}-(c-1))}=\frac{1}{2d}\sum_{0<t<d/c}\frac{1}{\frac{t}{d}}\frac{1}{1-\frac{ct}{d}}=\frac{c}{2d}\sum_{0<t<d/c}(\frac{1}{ct/d}+\frac{1}{1-ct/d}).
\]
Both sums are of $O(d\log d/c)$, so we have 
\[
\frac{1}{2d}\sum_{e=1}^{d-1}\frac{1}{(1-e/d)}\frac{1}{\{ce/d\}}=O(\log d).
\]
For the second sum 
\[
\frac{1}{2d}\sum_{e=1}^{d-1}\frac{1}{e/d}\frac{1}{(\{ce/d\})}=\frac{1}{2d}\sum_{p=0}^{c-1}\sum_{p<ce/d<p+1}\frac{1}{\frac{e}{d}}\frac{1}{(\frac{ce}{d}-p)},
\]
the part which $p\ge1$ is bounded by 
\[
\frac{1}{2d}\sum_{p=1}^{c-1}\sum_{p<ce/d<p+1}\frac{1}{\frac{e}{d}}\frac{1}{(\frac{ce}{d}-p)}=O(\frac{1}{2d}\sum_{p=1}^{c-1}\frac{1}{\frac{p}{c}}\cdot\frac{d}{c}\log d)=O(\frac{1}{2d}\frac{c}{\log c}\frac{d}{c}\log d)=O(\log d).
\]
The part for $p=0$ gives 
\[
\frac{1}{2d}\sum_{0<e<d/c}\frac{1}{e/d}\cdot\frac{1}{ce/d}=\frac{d}{2c}\sum_{0<e<d/c}\frac{1}{e^{2}}=\frac{d\pi^{2}}{12c}(1+O(\frac{c}{d}))=\frac{d\pi^{2}}{12c}+O(1).
\]
Combining all those results give 
\begin{align*}
R(c,d) & =\sum_{k=1}^{d-1}\log(2\sin(\pi k/d))\log(2\sin(\pi ck/d))\\
 & =\frac{\pi^{2}}{12}\frac{d}{c}-(\log d)^{2}+O(\log d)
\end{align*}
as required. \end{proof} 

For the case $d<c$, it will give
\[
R(c,d) = R(c \, \mathrm{mod} \, d) = \frac{\pi^2}{12} \frac{d}{(c \, \mathrm{mod} \, d)} - (\log d)^2 + O( \log d).
\]
where $(c \, \mathrm{mod} \, d)$ is remainder of $c$ by $d$.
Keeping this in mind, we combine theorem 1 \ref{thm:explicit+} and Lemma 4.1 to obtain
\begin{align*}
M_{+}(f,c) & =\frac{\phi(f)}{f^{2}}\sum_{d\vert f}\mu\left(\frac{f}{d}\right)d\left(\left(\log\frac{f}{d}\right)^{2}R(c,d)\right)-\left(\frac{\phi(f)}{f}\sum_{p\vert f}\frac{\log p}{p-1}\right)^{2}\\
 & =\frac{\phi(f)}{f^{2}}\sum_{d\vert f}\mu\left(\frac{f}{d}\right)d\left(\left(\log\frac{f}{d}\right)^{2}-\left(\log d\right)^{2}+\frac{\pi^{2}}{12c}d+O(\log d)\right)-\left(\frac{\phi(f)}{f}\sum_{p\vert f}\frac{\log p}{p-1}\right)^{2}\\
& + \frac{\phi(f)}{f^2} \sum_{\substack{d \vert f \\ d < c}} \mu\left(\frac{f}{d}\right) d \left(  \frac{\pi^2}{12} \frac{d}{(c \, \mathrm{mod} \, d)} - \frac{\pi^2}{12}\frac{d}{c}  + O (\log d) \right) \\
& =\frac{\phi(f)}{f^{2}}\sum_{d\vert f}\mu\left(\frac{f}{d}\right)d\left(\left(\log\frac{f}{d}\right)^{2}-\left(\log d\right)^{2}\right)+\frac{\phi(f)}{f^{2}}\frac{\pi^{2}}{12c}\sum_{d\vert f}\mu\left(\frac{f}{d}\right)d^{2}\\
 & -\left(\frac{\phi(f)}{f}\sum_{p\vert f}\frac{\log p}{p-1}\right)^{2}+O\left(\frac{\phi(f)}{f^{2}}\sum_{d\vert f}|\mu\left(\frac{f}{d}\right)|d\log d  +  \frac{\phi(f)}{f^2} \sum_{\substack{d \vert f \\ d<c}} |\mu \left( \frac{f}{d} \right) | d^2  \right).
\end{align*}
One can see that the error term can be bounded as 
\[
\frac{\phi(f)}{f^{2}}\sum_{d\vert f}|\mu\left(\frac{f}{d}\right)|d\log d\le\frac{\phi(f)}{f}\sum_{d\vert f}\frac{|\mu(f/d)|}{f/d}\log f=\log f\prod_{p\vert f}\left(1-\frac{1}{p}\right)\left(1+\frac{1}{p}\right)<\log f
\]
and
\[
 \frac{\phi(f)}{f^2} \sum_{\substack{d \vert f \\ d<c}} |\mu \left( \frac{f}{d} \right) | < \frac{1}{f} \sum_{d<c} d^2 < \frac{c^3}{f}
\]
and the second term contributes as the main term 
\[
\frac{\phi(f)}{f^{2}}\frac{\pi^{2}}{12c}\sum_{d\vert f}\mu\left(\frac{f}{d}\right)d^{2}=\frac{\pi^{2}}{12c}\phi(f)\sum_{d\vert f}\frac{\mu(f/d)}{(f/d)^{2}}=\frac{\pi^{2}}{12c}\phi(f)\prod_{p\vert f}\left(1-\frac{1}{p^{2}}\right).
\]
For the rest, 
\begin{align*}
 & \frac{\phi(f)}{f^{2}}\sum_{d\vert f}\mu\left(\frac{f}{d}\right)d\left(\left(\log\frac{f}{d}\right)^{2}-\left(\log d\right)^{2}\right)\\
 & =\frac{\phi(f)}{f}\sum_{d\vert f}\frac{\mu(f/d)}{f/d}\log f(2\log\frac{f}{d}-\log f)\\
 & =\frac{2\phi(f)\log f}{f}\sum_{d'\vert f}\frac{\mu(d')\log d'}{d'}-\frac{\phi(f)(\log f)^{2}}{f}\sum_{d'\vert f}\frac{\mu(d')}{d'}\\
 & =\frac{2\phi(f)\log f}{f}\sum_{d'\vert f}\frac{\mu(d')\log d'}{d'}-\frac{\phi(f)^{2}(\log f)^{2}}{f^{2}}\\
 & =-\left(\frac{\phi(f)\log f}{f}\right)^{2}-2\frac{\phi(f)\log f}{f}\cdot\frac{\phi(f)}{f}\sum_{p\vert f}\frac{\log p}{p-1}.
\end{align*}
Thus our $M_{+}(f,c)$ has asymptotic of 
\[
M_{+}(f,c)=\frac{\pi^{2}}{12c}\phi(f)\prod_{p\vert f}\left(1-\frac{1}{p^{2}}\right)-\frac{\phi(f)^{2}}{f^{2}}\left(\log f+\sum_{p\vert f}\frac{\log p}{p-1}\right)^{2}+O(c^3 f^{-1} + \log f).
\]
This proves Theorem \ref{thm:asymp+}.

\section{Conclusion}

Theorem \ref{thm:asymp overall} is a direct consequence of Louboutin's result in \cite{key-5} and our Theorem \ref{thm:asymp+}. If we recall 
\[
M_{-}(f,c) = \frac{\pi^2}{12c} \frac{\phi(f)^2}{f} \left( \prod_{p \vert f} (1+\frac{1}{p}) - \frac{3c}{f} \right) - \frac{\pi^2 \phi(f)}{2c f^2} \sum_{d \vert f} d \mu(f/d) S(c,d),
\]
since $S(c,d)$ is the quantity depending only on $d$ mod $c$, one can see that $S(c,d)=O(1)$. We also note that
\[
\frac{\pi^2 \phi(f)}{2c f^2} \sum_{d \vert f} d \mu(f/d) S(c,d) = O( \frac{\phi(f)}{f} \sum_{d \vert f} \frac{|\mu(f/d)|}{f/d}) = O(\prod_{p \vert f} (1-\frac{1}{p})(1+\frac{1}{p}) ) = O(1)
\]
and
\[
\frac{\pi^2}{12c} \frac{\phi(f)^2}{f} \frac{3c}{f} = O(1).
\]
Hence
\[
M_{-}(f,c) = \frac{\pi^2}{12c} \frac{\phi(f)^2}{f} \prod_{p \vert f} (1+\frac{1}{p}) + O(1) = \frac{\pi^2}{12c} \phi(f) \prod_{p \vert f} \left( 1 - \frac{1}{p^2} \right) + O(1).
\]
Thus with our Theorem \ref{thm:asymp+} we can conclude that 
\[
M(f,c)=\frac{\pi^{2}}{6c}\phi(f)\prod_{p\vert f}\left(1-\frac{1}{p^{2}}\right)-\frac{\phi(f)^{2}}{f^{2}}\left(\log f+\sum_{p\vert f}\frac{\log p}{p-1}\right)^{2}+O(\log f)
\]
as stated in Theorem \ref{thm:asymp overall}.
\begin{acknowledgements} The authours would like to thank Min-Soo Kim and Roger Heath-Brown for helpful discussions.
\end{acknowledgements}

\end{document}